\documentclass[12pt,a4paper]{amsart}
 
\usepackage[euler-digits]{eulervm}
\usepackage{graphicx}

\usepackage{amsfonts, amsthm, amssymb, amsmath, stmaryrd}
\usepackage{mathrsfs,array}
\usepackage{eucal,times,color,enumerate,accents}
\usepackage{url}
\usepackage{float}
\usepackage{pbox}
\usepackage[headings]{fullpage}

\usepackage{ytableau}
\usepackage{color}
\usepackage{mathrsfs}
\usepackage{amssymb}
\usepackage{tikz}
\usepackage{tikz-cd}
\usepackage{bm}
\usepackage{enumerate}
\usetikzlibrary{calc}
\usetikzlibrary{fadings}

\newtheorem{theorem}{Theorem}[section]
\newtheorem*{theorem*}{Theorem}
\newtheorem{proposition}[theorem]{Proposition}
\newtheorem{lemma}[theorem]{Lemma}
\newtheorem{corollary}[theorem]{Corollary}
\newtheorem{definition}[theorem]{Definition}

\newtheorem{remark}[theorem]{Remark}

\newtheorem*{lemma*}{Lemma}
\newtheorem*{remark*}{Remark}
\newtheorem*{example*}{Example}

\newtheorem{innercustomthm}{Main Theorem}
\newenvironment{customthm}[1]
  {\renewcommand\theinnercustomthm{#1}\innercustomthm}
  {\endinnercustomthm}
  
\newcommand{\spec}{\operatorname{Spec}}  

\DeclareMathAlphabet{\mathpzc}{OT1}{pzc}{m}{it}
\makeatletter
\def\blfootnote{\xdef\@thefnmark{}\@footnotetext}
\makeatother

\title[Motivic Hilbert Zeta]{Motivic Hilbert zeta functions of curves
 are rational}
\author[ Bejleri--Ranganathan--Vakil]{{\larger D}{\smaller ori}\  {\larger B}{\smaller ejleri} \ \ \ {\larger D}{\smaller hruv}\  {\larger R}{\smaller anganathan} \ \ \ {\larger R}{\smaller avi}\  {\larger V}{\smaller akil}}

\address{Dori Bejleri, Department of Mathematics\\
Brown University\\ Providence RI 02913 USA}
\email{\href{mailto:dbejleri@math.brown.edu}{dbejleri@math.brown.edu}}

\address{Dhruv Ranganathan, Department of Mathematics\\
Massachusetts Institute of Technology\\
Cambridge MA 02138 USA}
\email{\href{mailto:dhruvr@mit.edu}{dhruvr@mit.edu}}

\address{Ravi Vakil, Department of Mathematics\\
Stanford University\\ Stanford CA 94305 USA}
\email{\href{mailto:vakil@math.stanford.edu}{vakil@math.stanford.edu}}

\date{\today}

\usepackage[backref]{hyperref}
\hypersetup{
  colorlinks   = true,          
  urlcolor     = purple,          
  linkcolor    = blue,          
  citecolor   = purple             
}

\newcommand{\mb}[1]{\mathbb{#1}}
\newcommand{\mf}[1]{\MakeUppercase{#1}}
\newcommand{\mc}[1]{\mathcal{#1}}

\newcommand{\Spec}{\text{Spec}}

\newcommand{\Hilb}{\operatorname{Hilb}}
\newcommand{\Sym}{\operatorname{Sym}}

\newcommand{\supp}{\operatorname{supp}}
\newcommand{\F}[1]{F^{\underline{#1}}}
\newcommand{\tF}[1]{\tilde{F}^{\underline{#1}}}

\def\ZZ{{\mathbb Z}}

\def\CC{{\mathbb C}}

\begin{document}

\begin{abstract}
The motivic Hilbert zeta function of a variety $X$ is the generating
function for classes in the Grothendieck ring of varieties of Hilbert
schemes of points on $X$. In this paper, the motivic Hilbert zeta
function of a reduced curve is shown to be rational. 
\end{abstract}
\maketitle

\section{Introduction} 

Let $k$ be an algebraically closed field and $K_0(\mathrm{Var}_k)$ be the Grothendieck ring
of varieties over $k$. Suppose $X$ is a variety over $k$. The
\textit{motivic zeta function of $X$} is defined as the power series
\[
Z^{\mathrm{mot}}_X(t) := \sum_{d\geq 0} [\mathrm{Sym}^d(X)] t^d \ \in \ K_0(\mathrm{Var}_k)\llbracket t\rrbracket
\]
(where $[\Sym^0(X)] = 1$ by convention). For the remainder of this paper, $X$ will be a curve. In~\cite{kapranov}, Kapranov observed that if $X$ is a
{\em smooth} curve, $Z_X^{\mathrm{mot}}(t)$ 
 is a \textit{rational} function in $t$. The zeta function is a rich invariant of $X$, and when $X$ is defined over a finite field, specializes to the Weil zeta function via the point-counting measure. When the curve $X$ is singular, one still expects $Z^{\mathrm{mot}}_X(t)$ to be a rational function in $t$, see for instance~\cite[Corollary 30]{litt} for strong results in this direction. However, $Z_X^{\mathrm{mot}}(t)$ is not sensitive to the singularities of $X$. For example when $X$ is a cuspidal cubic, $Z_X^{\mathrm{mot}}(t) = Z_{\mb{P}^1}^{\mathrm{mot}}(t)$. 

An invariant that is more sensitive to the singularities of $X$ is the \textit{motivic Hilbert zeta function}:
\[
Z_X^{\mathrm{Hilb}}(t):= \sum_{d\geq 0} [\mathrm{Hilb}^d(X)]t^d \ \in \ 1 + tK_0(\mathrm{Var}_k)\llbracket t\rrbracket.
\]
Here, $\mathrm{Hilb}^d(X)$ is the Hilbert scheme parametrizing length
$d$ subschemes of $X$. When $X$ is smooth, $\mathrm{Hilb}^d(X)$
coincides with $\mathrm{Sym}^d(X)$, so the Hilbert and usual motivic
zeta functions coincide. However, when $X$ is singular,
$\mathrm{Hilb}^d(X)$ contains information about subschemes supported
on the singularities. For instance, if $X = \mathrm{Spec}(k[ x,y
]/(xy))$ is a nodal curve with singular point $p \in X$, one can check
that $\mathrm{Hilb}_p^2(X) \cong \mb{P}^1$ where $\mathrm{Hilb}^d_p(X)
\subset \mathrm{Hilb}^d(X)$ is the locus parametrizing subschemes
supported on the singularity (see Section \ref{sec:local}). More
generally, if $p \in X$ is a singular point, $\mathrm{Hilb}^2_p(X)
\cong \mb{P}(T_pX)$ is the projectivization of the Zariski tangent
space of the singularity.  In particular, even $\mathrm{Hilb}^2(X)$
can be of arbitrarily large dimension.  The main result of this paper is the following.

\begin{customthm}{\!\!}
Let $X$ be a reduced curve over an algebraically closed
field $k$. Then $Z^{\mathrm{Hilb}}_X(t)$ is a rational function of $t$, with constant term $1$. 
\end{customthm}

As before, if $X$ is defined over a finite field, by passing to point counting over $\mathbb F_q$ we obtain a generalization of the rationality part of the Weil conjectures
for curves, to the case of Hilbert zeta functions. We note that even after passing to the Euler characteristic specialization, the result appears to be new. The main theorem
was known in special cases from work of others. When the singularities
of $X$ are planar, the result is implicit in work of Maulik and
Yun~\cite{MY14} and for Gorenstein curves when $k = \CC$, it is proved
by Migliorini and Shende~\cite[Proposition 16]{MS11}. For curves that
have unibranch singularities, one can deduce rationality from a result of Pfister and Steenbrink~\cite{ps} that the punctual Hilbert schemes of a unibranch singularity become isomorphic as the number of points goes to infinity. We establish a similar stabilization for multibranch singularities. 

When $X$ is planar, the Hilbert zeta function is closely related to
knot  invariants, \cite{OS12}. One might hope that for non-planar curves $X$, the Hilbert zeta function is still related to such structures, for example the refined invariants defined by Aganagic and Shakirov~\cite{AS11}.   Furthermore, it is natural to hope for a more conceptual or geometric explanation for the rationality proved here.

Our approach to the proof of the main result is as follows. Given a
singular point, we stratify the Hilbert scheme based on the lengths of
the pullbacks of the universal subscheme to branches of the
normalization. We then show that these lengths vary in a controlled
manner, based only on the singularity. We use this to show that each
of these strata stabilize, for large enough degree, inside an
appropriate Grassmannian, mimicking the construction of the Hilbert
scheme of points (see for instance~\cite[Part 3]{FGAx}). The stabilization is captured by the statement of Corollary~\ref{cor:dim}. In the last section we illustrate in an example how these methods may be used to compute the Hilbert zeta function explicitly. 

We work over an algebraically closed field for simplicity.   When $k$
is not algebraically closed, $X$ may fail to have a rational point and
the standard argument for rationality~\cite{kapranov} is not
sufficient, even when $X$ is smooth. Nonetheless, even without a
rational point, Litt \cite{litt}  has shown  that the motivic zeta function is still
rational. The arguments in the present paper generalize without substantial changes when the singular points of $X$ are $k$-rational. When the singularities are not $k$-rational, the methods here may be adapted by applying the Cohen structure theorem at the singularities~\cite[\href{http://stacks.math.columbia.edu/tag/0323}{Tag 0323}]{stacks-project}, with some additional careful bookkeeping. We leave these adaptations to the reader. 

\subsection*{Future prospects}
The class of the Hilbert scheme of a smooth
surface is studied in~\cite{dCL02,G01}, and a beautiful quasimodular
formula for the Betti numbers was obtained by G\"ottsche
in~\cite{G90}. A study of the Hilbert zeta function for singular surfaces is a
natural next step in view of the results here, and first steps in this direction are taken in~\cite{GNS15}. A more concrete goal
would be to establish rationality for generically non-reduced curves,
where our explicit methods do not seem to generalize in a
straightforward way. Since this paper first appeared on the ar$\chi$iv, the second author has used the approach in this paper to show that the Euler characteristic Hilbert zeta function is constructible in families of curves~\cite{Bej18}. 

\subsection*{Acknowledgments} This project began at 2015 Arizona Winter School on the \textit{Arithmetic of Higher Dimensional Algebraic Varieties}. We thank the organizers and our project group for creating an exciting mathematical environment. Important progress was made when D.B. visited Yale University in Fall 2015. We are especially grateful to Dan Abramovich for encouragement along the way, as well as for remarks on a previous draft. The paper has also benefited from conversations with Shamil Asgarli, Daniel Litt, Davesh Maulik, and Luca Migliorini. D.R. was a Ph.D. student at Yale University when this project started, and warmly acknowledges friends and colleagues there for their support. The final text benefited from the thoughtful comments of an anonymous referee. \\

\subsection*{Funding} D.B. was supported by NSF grant DMS-1162367 (PI: Dan Abramovich), D.R. was supported by NSF grants CAREER DMS-1149054 (PI: Sam Payne) and DMS-128155 (Institute for Advanced Study), and R.V.  was supported by NSF grants DMS-1500334 and DMS-1100771.

\section{Hilbert schemes of points}

For $X$ a quasiprojective variety, the Hilbert scheme $\Hilb^d(X)$ is the moduli space for flat families of length $d$ subschemes of $X$. Using the identification between length $d$ subschemes $Z \subset X$ and ideal sheaves $\mathcal J$ with $\mathrm{colength}(\mathcal J) := \mathrm{length}(\mathcal O_X/\mathcal J) = d$, we will often represent the closed points of $\Hilb^d(X)$ by the corresponding ideals. 

There is a well defined Hilbert-Chow morphism (see, for example,
\cite[Chapter 7]{FGAx})
$$
h : \Hilb^d(X) \to \Sym^d(X)
$$
sending a subscheme to its support:
$$
[\mathcal J] \mapsto \sum_{p \in \mathrm{Supp}(\mathcal O_X/\mathcal J)} \mathrm{length}(\mathcal O_{X,p}/\mathcal J_p) [p].
$$
When $X$ is a smooth curve, $h$ is an isomorphism. 

\subsection{Reduction to a local calculation}\label{sec:local} We first reduce the proof of our main theorem to a local calculation at the singularities of the curve. Let $Y \subset X$ be a closed $k$-subvariety. Then $\Sym^d(Y) \subset \Sym^d(X)$ is a closed subvariety and we define $\Hilb^d(X,Y)$ the \emph{Hilbert scheme with support in $Y$} as the scheme theoretic preimage $h^{-1}(\Sym^d(Y))$ by the Hilbert-Chow morphism $h : \Hilb^d(X) \to \Sym^d(X)$. Set theoretically, $\Hilb^d(X,Y) \subset \Hilb^d(X)$ consists of length $d$ subschemes $Z \subset X$ with support $\supp(\mathcal O_Z)$ contained in $Y$. 

We define the \emph{motivic Hilbert zeta function with support in} $Y$ as:

\[
Z_{Y\subset X}^{\Hilb}(t) := \sum_{d \geq 0} [\Hilb^d(X,Y)]t^d \in 1 + tK_0(\mathrm{Var})\llbracket t \rrbracket
\]

The Hilbert zeta function respects the scissor relations on $X$ in the following sense.

\begin{lemma} Let $Y \subset X$ a closed subset with open complement $U \subset X$. Then
$$
Z_{X}^{\Hilb}(t) = Z_{U}^{\Hilb}(t) \cdot Z_{Y\subset X}^{\Hilb}(t)
$$
\end{lemma}

\begin{proof}  Stratify $\Hilb^d(X)$ into locally closed subsets
\[
\Hilb^d(X) = \bigsqcup_{i + j = d} \Hilb^i(U) \times \Hilb^j(X,Y).
\]
Here $\Hilb^i(U) \times \Hilb^j(X,Y)$ is the stratum consisting of subschemes of $X$ of length $d$, such that the length of the subscheme supported on $Y$ is exactly $j$. This implies that
\[
[\Hilb^d(X)] = \sum_{i + j = d} [\Hilb^i(U)]\cdot[\Hilb^j(X,Y)]
\]
in $K_0(\mathrm{Var})$ and the result follows. 
\end{proof}

\begin{corollary} 
Let $X$ be a reduced curve over $k$ with possibly singular points $p_1, \ldots, p_l$. Then
\[
Z_{X}^{\Hilb}(t) = Z_{X^{sm}}^{\Hilb}(t) \prod_{i = 1}^l Z_{p_i \subset X}^{\Hilb}(t).
\]
where $X^{sm}$ denotes the smooth locus of $X$. 
\end{corollary}

Upon applying Kapranov's theorem~\cite[Theorem 1.1.9]{kapranov}  in conjunction with the identification of $\Hilb^d(X)$ with $\mathrm{Sym}^d(X)$ for $X$ a smooth curve, we see that $Z_{X^{sm}}^{\Hilb}(t)$ is a rational function. Thus, the proof of the main theorem will follow from the following result:

\begin{theorem}\label{thm:1} Let $(X,0)$ be a reduced
  curve with singular point $0 \in X$. Then $Z_{0 \subset X}^{\Hilb}(t)$ is a rational function in $t$ with denominator $(1 - t)^s$ where $s$ is the number of branches of the singularity $(X,0)$.
\end{theorem} 

We will refer to the pair $(X,0)$ as a curve singularity and
$\Hilb^d(X,0)$ as the punctual Hilbert scheme of $(X,0)$. Note that $\Hilb^d(X,0)$ depends only on the completed local ring $R = \widehat{\mathcal O_{X,0}}$. In fact there is a natural identification of the punctual Hilbert scheme
$$
\Hilb^d(X,0) = \{[\mathcal J]\ | \ \mathcal J \subset R, \ \mathrm{colength}(\mathcal J) = d\}
$$
as a parameter space for colength $d$ ideals in $R$.
\subsection{A stratification of the punctual Hilbert scheme}

Let $R$ be a reduced complete local ring of dimension $1$ over $k$ with residue field $k$, i.e., the completed local ring of the germ of a $k$-rational curve singularity. Let $\tilde R$ denote its normalization. If $X = \spec R$ is an $s$-branched curve singularity, then we have an isomorphism
\[
\tilde X:=\spec \tilde R \cong \spec\left(k\llbracket x_1\rrbracket \times \cdots \times k\llbracket x_s\rrbracket \right).
\]
  Let $B_i := \Spec k\llbracket x_i\rrbracket$ be the $i^{\mathrm{th}}$ branch of $\tilde X$ and $\varphi_i:B_i\to \spec(R)$ be the normalization map restricted to this branch. 

Let $\Hilb^d(X,0)$ denote the punctual Hilbert scheme of points on $X$. Let $\underline{\bm{a}} = (a_1,\ldots, a_s)\in \mathbb{N}^s$ be a vector of non-negative integers. For a length $d$ subscheme defined by an ideal $\mathcal I$, let $[\mathcal I]$ denote the corresponding point in $\Hilb^d(X,0)$, over which the universal subscheme is $Z_\mathcal{I} := \spec(R/\mathcal I)$. Define the subset $\Hilb^{d,\underline{\bm{a}}}(X,0)$ to be the locus
\[
\Hilb^{d,\underline{\bm{a}}}(X,0) := \{[\mathcal{I}]\in \Hilb^d(X,0): \textrm{for all $i$,  } \mathrm{length}(\varphi_i^\star(Z_{\mathcal{I}})) = a_i\}. 
\]

\begin{definition}\label{def:branch-length}
Given a subscheme $Z_{\mathcal I}$ in the notation above, the vector $\underline{\bm{a}}$ will be referred to as the \textit{branch-length vector} of the subscheme.
\end{definition}

The branch-length vector stratifies the punctual Hilbert scheme in the following sense.

\begin{proposition}\label{strata}
Let $d$ and $\underline{\bm{a}}$ be as above. The subset
$\Hilb^{d,\underline{\bm{a}}}(X,0)$ is a locally closed subscheme of
$\Hilb^d(X,0)$ (possibly empty).
\end{proposition}

\begin{proof} It suffices to prove that for any single branch $\varphi : B \to X$, 
$$
\Hilb^{d,e}(X,0) := \{[\mathcal{I}] \in \Hilb^d(X,0): \mathrm{length}(\varphi^*(Z_{\mathcal{I}})) = e\}
$$
\noindent is locally closed. Consider the universal flat family
$$
\begin{tikzcd}
\mathcal{Z} \arrow[hook]{r} \arrow{d} & X \times \Hilb^d(X,0) \\
\Hilb^d(X,0) & 
\end{tikzcd}
$$
\noindent of closed subschemes of $X$ over $\Hilb^d(X,0)$. Pulling back this diagram along \
\[
(\varphi,\mathrm{id}) : B \times \Hilb^d(X,0) \to X \times \Hilb^d(X,0)
\]
gives us a diagram 
$$
\begin{tikzcd}
\varphi^*\mathcal{Z} \arrow[hook]{r} \arrow{d}{\pi} & B \times \Hilb^d(X,0) \\
\Hilb^d(X,0) & 
\end{tikzcd}.
$$
\noindent of closed subschemes of $B$ over $\Hilb^d(X,0)$. 

The morphism $\pi$ is finite, so the function $[\mathcal{I}] \mapsto \mathrm{length}(\pi^{-1}[\mathcal{I}])$ is upper semicontinuous~\cite[Theorem III.2.8]{hartshorne}. Thus we can stratify $\Hilb^d(X,0)$ into a disjoint union of locally closed $S_e \subset \Hilb^d(X,0)$ over which $\pi$ is finite of constant degree $e$. It remains to check that $S_e = \Hilb^{d,e}(X,0)$. Indeed for any $[\mathcal{I}] \in \Hilb^d(X,0)$, $\pi^{-1}[\mathcal{I}] = \mathrm{Spec}(R/\mathcal{I} \otimes_R B) = \varphi^*(\mathrm{Spec}(R/\mathcal{I}))$, so that $\Hilb^{d,e}(X,0)$ is precisely the locus over which $\pi$ has constant degree $e$.

\end{proof}

\subsection*{Remark} The stratification in the proof above is in fact the set-theoretic version of the  flattening stratification for $\pi$. Observe that any finite morphism of constant degree is flat, so the restriction of $\pi$ over each $S_e$ is flat. On the other hand, finite flat morphisms have constant degree. In particular, it follows from the universal property of the flattening stratification that $\Hilb^{d,\underline{\bm{a}}}(X,0)$ is a moduli space for length $d$ subschemes $Z \subset X$ with $\mathrm{length}(\varphi_i^*(Z)) = a_i$. See~\cite[Tag 052F]{stacks-project} for details on this stratification.

\section{The geometry of singular curves}

Consider a reduced curve singularity $X$ with $s$ branches. Let $\tilde X\to X$ be the normalization, opposite to the finite extension 
\[
R\hookrightarrow \tilde R \cong \prod_{i = 1}^s k\llbracket x_i \rrbracket
\]
of rings. We will identify $R$ with a subring of $\tilde R$, and write $R_i$ for the coordinate ring of the branch $i^{th}$ branch $X_i$. In other words, $R_i \subset k\llbracket x_i \rrbracket$ is a finite ring extension corresponding to the $i^{th}$ branch $\varphi_i : B_i \to X_i \subset X$ of the normalization.

\begin{enumerate}[(1)]
\item Let 
\[
\delta := \mathrm{dim}_k \prod_{i = 1}^s k\llbracket x_i \rrbracket/R
\]
be the \textit{$\delta$-invariant} of $X$. Similarly, we denote
by $\delta_i$ the $\delta$-invariant $\dim_k k\llbracket x_i \rrbracket/R_i$ of the $i^{th}$ branch.

\item Let 
\[
\mathfrak{c} := \mathrm{Ann}_{R}(\tilde{R}/R)
\]
be the \textit{conductor ideal}. This an ideal of both $\tilde{R}$ and $R$. In particular $\mathfrak{c}$ is generated by monomials, say $x_i^{c_i}$, as an ideal of $R$. It's clear from the definition that $c_i$ is the smallest positive integer such that for all $n \geq c_i$, $x_i^n \in R$. We will refer to $c_i$ as the conductor of the $i^{th}$ branch, denote by
$$
C := \dim \tilde{R}/\mathfrak{c} = \sum_{i = 1}^s c_i
$$
the conductor of $X$, and by $\underline{c} = (c_1, \ldots, c_s)$ the conductor branch-length vector.

\end{enumerate} 

We will need the following result of Schwede:

\begin{proposition}\label{prop:pushout}(Schwede \cite{schwede}) Let $R \subset \tilde{R}$ be the normalization of a reduced ring and let $\mathfrak{c}$ be the conductor ideal. Then 
$$
\begin{tikzcd}
\spec \tilde{R}/\mathfrak{c} \arrow{d} \arrow[hook]{r} & \spec \tilde{R} \arrow{d} \\ \spec R/\mathfrak{c} \arrow[hook]{r}  & \spec R
\end{tikzcd}
$$
is a pushout diagram of affine schemes. 
\end{proposition}

\begin{proof} We want to show that $R \hookrightarrow \tilde{R}$ is the pullback of $R/\mathfrak{c} \hookrightarrow \tilde{R}/\mathfrak{c}$ along the quotient $\tilde{R} \to \tilde{R}/\mathfrak{c}$. Let $A$ be this fiber product. There is a map $R \to A$ by universal property which is injective since the composition $R \to \tilde{R}$ is. Let $(x,\bar{y}) \in A$ so that $x \in \tilde{R}$, $\bar{y} \in \spec R/\mathfrak{c}$ and $x + \mathfrak{c} = \bar{y}$. Then $x - y \in \mathfrak{c} \subset R \subset \tilde{R}$ where $y \in R$ is some lift of $\bar{y}$ so $x \in R$ and $R \to A$ is surjective. 

\end{proof} 

We use this to show that any reduced curve singularity appears as the unique singularity of a connected rational curve with all irreducible components unibranch. 

\begin{corollary}\label{cor:pushout} Let $R$ be the completed local ring of an $s$-branched curve singularity. There exists a connected affine curve $Y$ with normalization $\bigsqcup_{i = 1}^s \mb{A}^1$ and unique singular point $0 \in Y$ such that $\widehat{\mathcal{O}_{Y,0}} \cong R$ and the diagram
$$
\begin{tikzcd}
\spec \left(\prod_{i = 1}^s k\llbracket x_i \rrbracket\right) \arrow[d] \arrow[r] & \bigsqcup_{i = 1}^s \mb{A}^1 \arrow[d] \\ \spec R \arrow[r] & Y
\end{tikzcd}
$$
commutes.
\end{corollary}

\begin{proof} The composition $\prod_{i = 1}^s k[x_i] \to \prod_{i = 1}^s k\llbracket x_i \rrbracket \to \tilde{R}/\mathfrak{c}$ is evidently surjective and so induces a closed embedding $\Spec \tilde{R}/\mathfrak{c} \hookrightarrow \bigsqcup_{i = 1}^s \mb{A}^1$. Now we define $Y$ to be the pushout of the diagram
$$
\begin{tikzcd}
\spec \tilde{R}/\mathfrak{c} \arrow[hook]{r} \arrow{d} & \bigsqcup_{i = 1}^s \mb{A}^1 \\ \spec R/\mathfrak{c} &. 
\end{tikzcd}
$$
which exists since everything is affine. By Proposition \ref{prop:pushout} and the universal property of pushouts, there exists a unique $\spec R \to Y$ making the diagram in the statement commute. Finally, the induced map $\widehat{\mathcal{O}_{Y,0}} \to R$  is an isomorphism since completion commutes with fiber products of rings.  

\end{proof}

\subsection{The branch-length filtration and graded degenerations} Let $R \subset \tilde{R}$ as above be the completed local rings of an $s$-branched reduced curve singularity $X$ and its normalization $\tilde{X}$. Let $A = \mc{O}_Y \subset R$ be the coordinate ring of a rational curve $Y$ as constructed in Corollary \ref{cor:pushout} so that the normalization $\tilde{A} = \prod k[x_i] \subset \prod k\llbracket x_i \rrbracket = \tilde{R}$. 

Denote by $v_i : \tilde{R} \to \mathbb{N}$ the composition of the projection onto $k\llbracket x_i \rrbracket$ with the valuation on $k\llbracket x_i \rrbracket$. This gives the order of vanishing of a function along the $i^{th}$ branch of the normalization. 

\begin{definition} We define an $\mb{N}^s$-filtration $\tilde{F}^\bullet$ on $\tilde{R}$ by
$$
\tF{a} := \{f \in \tilde{R} \ | \ v_i(f) \geq a_i \ i = 1, \ldots, s\} \subset \tilde{R} 
$$
for $\underline{a} \in \mb{N}^s$. The restriction to $R \subset \tilde{R}$ is denoted by
$$
\F{a} := \{f \in R \ | \ v_i(f) \geq a_i, \ i = 1,
\ldots, s\} \subset R.
$$
\end{definition}

Equivalently, $\tF{a}$ is the ideal of $\tilde{R}$ generated by $x_i^{a_i}$ for $i = 1, \ldots, s$ and $\F{a} = \tF{a} \cap R$ is the ideal of functions on $X$ vanishing to order at least $a_i$ along $B_i$. Note in particular that the conductor $\mathfrak{c} = \F{c} = \tF{c}$ and $\F{a} = \tF{a}$ if and only if $\tF{a} \subset \mathfrak{c}$. 

Finally, note that $\tF{\bullet}$ restricts to filtrations on $A$ and $\tilde{A}$ as well. We will abuse notation and also denote these by $\tilde{F}^\bullet$ and $F^\bullet$ respectively where the meaning will be clear from context. 

Let $\underline{w} = (w_1, \ldots, w_s)$ be a vector of non-negative integers and denote by
$$
\underline{w} \cdot \underline{a} = \sum w_i a_i
$$
the usual inner product. Given such a weight vector $\underline{w}$, we obtain an $\mb{N}$-filtration $\tilde{F}^\bullet_{\underline{w}}$ on $\tilde{R}$ (resp. $\tilde{A}$) by 
$$
\tilde{F}^n_{\underline{w}} := \sum_{\underline{w}\cdot \underline{a} \geq n} \tF{a}.
$$
Denote by $F_{\underline{w}}^\bullet$ the restriction of $\tilde{F}_{\underline{w}}^\bullet$ to $R$ (resp. $A$). 

\begin{definition} The (extended) Rees algebra of an $\mb{N}$-filtered ring $(B,F^\bullet)$ is 
$$
\mc{R}ees(B,F^\bullet) := \sum_{n \in \ZZ} F^n t^{-n} \subset B[t,t^{-1}]
$$
where by convention, $F^n = B$ for $n \le 0$. \end{definition}

The Rees algebra has the following useful properties: 
\begin{itemize}
\item $\mc{R}ees(B, F^\bullet)$ is flat over $k[t]$ by Lemma \ref{lem:flat} below; 
\item $\mc{R}ees(B, F^\bullet)/(t - a)\mc{R}ees(B,F^\bullet) \cong B$ for $a \neq 0$;
\item $\mc{R}ees(B,F^\bullet)/t\mc{R}ees(B,F^\bullet) \cong gr_{F^\bullet} B$ the associated graded ring. 
\end{itemize} 
These can be checked directly from the definition. We refer the reader to~\cite[Chapter 6]{Eis} for details on the Rees algebra.

We use this to construct an equinormalizable degeneration of a reduced curve singularity to a non-normal toric singularity. We take inspiration from Gr\"obner theory -- by choosing a sufficiently generic weight vector $\underline w$, the Rees algebra construction allows us to construct a degeneration whose special fiber is a monomial subring generated by the ``$\underline w$-leading terms'', stated formally below. Let $Y = \spec A$ be as in Corollary \ref{cor:pushout}.

\begin{theorem}\label{thm:monomial} There exists a flat family of connected affine curves $\mathscr{Y} \to \mb{A}^1$ with a section $\sigma : \mb{A}^1 \to \mathscr{Y}$ such that $\mathscr{Y} \setminus \sigma$ is smooth over $\mb{A}^1$, the $\delta$-invariant and number of branches of the singularity $(\mathscr{Y}_b, \sigma(b))$ are constant for all $b \in \mb{A}^1$, $\mathscr{Y}_b \cong Y$ for $b \neq 0$, and $\mathscr{Y}_0 = \spec A_0$ where $A_0 \subset \tilde{A} = \prod_{i = 1}^s k[x_i]$ is a monomial subring. 

\end{theorem}

\begin{proof} Observe that the filtrations $\tF{a} = \F{a}$ for all $a_i \geq c_i$, that is, $A$ and $\tilde{A}$ agree in degrees above the conductors. In particular, there are only finitely many degrees $a_{ij}$ such that $x_i^{a_{ij}} \in \tilde{A}$ but not $A$. Pick a positive integral weight vector such that $\underline{w}$ such that $\sum_i a_{ij} w_i$ are distinct integers for all choices of such $j$. That is, each monomial in low degree of $\tilde{A}$ is in a $1$-dimensional graded piece of the split filtration $\tilde{F}^\bullet_{\underline{w}}$. \footnote{By a split filtration, we mean one induced by a direct sum decomposition.} 

Now let $\mathscr{Y} = \spec \mc{R}ees(A, F_{\underline{w}}^\bullet) \to \mb{A}^1$. This is a flat family with all fibers away from zero isomorphic to $Y$ and central fiber $\mathscr{Y}_0 = \spec gr_{F^\bullet_{\underline{w}}} A =: \spec A_0$. The inclusion $\mc{R}ees(A,F^\bullet_{\underline{w}}) \subset A[t,t^{-1}]$ induces a dominant morphism $Y \times \mb{G}_m \to \mathscr{Y}$. Let $\sigma \subset \mathscr{Y}$ be the smallest closed subscheme through which the the singular locus $0 \times \mb{G}_m \subset Y \times \mb{G}_m$ factors -- that is, the scheme theoretic image of the singular locus~\cite[\href{http://stacks.math.columbia.edu/tag/01R5}{Tag 01R5}]{stacks-project}. Then $\sigma$ is flat over $\mb{A}^1$ by Lemma \ref{lem:flat} below and it has generic degree $1$ so it is a section. 

As $A \subset \tilde{A}$ is a finite ring extension of filtered rings, there is an induced finite ring extension $A_0 \subset gr_{\tilde{F}^\bullet_{\underline{w}}} \tilde{A}$. But $\tilde{A}$ is already graded so the latter is just $\tilde{A}$ and $A_0 \subset \tilde{A}$ is the normalization. By construction, $\sigma(0)$ is the vanishing locus of the ideal $\tilde{F}^{1,\ldots, 1} \cap A_0$ so the normalization is an isomorphism on the complement $\mathscr{Y}_0 \setminus \sigma(0)$. In particular, $\mathscr{Y} \setminus \sigma$ is smooth as required. 

Furthermore, the normalization can be done in families. Indeed the isotrivial family $$\spec \mc{R}ees(\tilde{A}, \tilde{F}^\bullet_{\underline{w}}) \to \mathscr{Y}$$ is a simultaneous normalization. It follows that the number of branches and the $\delta$-invariant of $(\mathscr{Y},\sigma(b))$ is constant (\cite{teiss} \cite[Theorem 5.2.2]{bug}). 

Finally, $A_0 \subset \tilde{A}$ is a graded subalgebra and we chose the weight $\underline{w}$ so that the graded pieces in degrees smaller than the conductor are one dimensional spanned by monomials and $A_0$ and $A$ agree in degree larger than the conductors so $A_0$ must be generated by monomials.

\end{proof}

\begin{remark} A special case of Theorem \ref{thm:monomial} for planar unibranch curves is used by Goldin and Teissier (see \cite[Proposition 3.1]{gt}) in order to study simultaneous resolution of a family curve singularities. Recently Kaveh and Murata \cite{km} used Rees algebras to construct analagous toric degenerations of projective varieties. 
\end{remark} 

\begin{lemma}\label{lem:flat} Let $X \xrightarrow{f} Y \to \mb{A}^1$ be morphisms of schemes such that $Y$ is the scheme theoretic image of $f$ and $X \to \mb{A}^1$ is flat. Then $Y \to \mb{A}^1$ is flat. In particular, if $A \subset R$ is a $k[t]$-algebra extension and $R$ is flat over $k[t]$, then so is $A$. \end{lemma} 
\begin{proof} The associated points of $X$ map onto the associated points of its scheme theoretic image $Y$. A morphism to $\mb{A}^1$ is flat if and only if associated points all map to the generic point so the result follows. 
\end{proof}

Given an ideal $\mathcal I \subset A$, we can define an ideal sheaf $\mathscr{I}$ on $\mathscr{Y}$ by the intersection 
$$
\mathcal I A[t,t^{-1}] \cap \mc{R}ees(A,F^\bullet_{\underline{w}}).
$$
It is evident that $\mathcal I_0 := \mathscr{I}/t \mathscr{I}$ is the associated graded ideal of $A_0$. 

\begin{corollary}\label{cor:monomial} Suppose $Z \subset Y$ is a closed subscheme with ideal $\mathcal I$ and let $\mathscr{I}$ be as above. Then the closed subscheme $\mathscr{Z} \subset \mathscr{Y}$ cut out by $\mathscr{I}$ is flat over $\mb{A}^1$. Furthermore $\mathscr{Z}_b \cong Z$ for $b \neq 0$ and $\mathscr{Z}_0$ is a monomial subscheme. \end{corollary} 

\begin{proof} We need only check flatness as the rest follows from the definition of $\mathscr{I}$. By construction, $\mathscr{Z}$ is the scheme theoretic image of the constant family $Z \times \mb{G}_m$ under the dominant morphism $Y \times \mb{G}_m \to \mathscr{Y}$ so $\mathscr{Z}$ is flat over $\mb{A}^1$ by Lemma \ref{lem:flat}. \end{proof}

\section{Degree--branch-length bounds} 

Recall the definition of the branch-length vector of a subscheme in Definition~\ref{def:branch-length}. Our proof  proceeds in two main steps. In this section we show that the quantity $d - \sum a_i$ for which $\Hilb^{d,\underline{\bm{a}}}(X,0)$ is nonempty is uniformly bounded by the invariants of the singularity $(X,0)$. In the next section, we use these bounds to embed $\Hilb^{d,\underline{\bm{a}}}(X,0)$ into a Grassmannian and show that these locally closed subsets stabilize in the Grothendieck ring. \\

\begin{lemma}\label{lem:delta} Let $\mathcal J \subset R$ be a finite colength ideal and suppose there exist $f_i \in \mathcal J$ for $i = 1, \ldots, s$ with $v_i(f_i) = l_i$. Then 
\[
F^{l_1 + c_1, \ldots, l_s + c_s} = \tilde{F}^{l_1 + c_1, \ldots, l_s + c_s} \subset \mathcal J.
\]

\end{lemma}
\begin{proof} The equality between the two ideals is clear since they are both contained in the conductor. We claim that given $f_i \in \mathcal J$ with $v_i(f_i) = l_i$ then $x_i^{l_i + m} \in \mathcal J$ for all $m \geq c_i$. It suffices to check this one branch at a time so without loss of generality suppose $s = 1$. 

Up to scaling, we may write $f = x^l + g(x)$ where $g(x)$ is higher order terms. We have $x^mf = x^{l + m} + x^m g(x) \in \mathcal J$ for any $m \geq 2c$. In particular, $x^{l + m} + h(x) \in \mathcal J$ for some $h(x)$ of arbitrarily large order. Since $\mathcal J$ is finite colength, $x^n \in \mathcal J$ for all $n$ large enough so $x^{l + m} \in \mathcal J$.  

\end{proof} 

\begin{proposition}\label{prop:inclusions} Let $\mathcal I$ be the ideal of a closed subscheme $Z \subset X$ having length $d$ and branch-length vector $\underline{\bm{a}} = (a_1, \ldots, a_s)$. Then we have the inclusions
$$
F^{\underline{\bm{a}} + \underline{\bm{c}}} \subset \mathcal I \subset  \F{\bm{a}}.
$$
\end{proposition} 

\begin{proof} There is a morphism of $R$-modules $\mathcal I \to \mathcal I \tilde R_i$ given by composing the inclusion $\mathcal I \subset R \subset \tilde{R}$ with the projection $\tilde R \to \tilde R_i$. The image $\mathrm{im}(\mathcal I \to \mathcal I \tilde R_i)$ generates $\mathcal I \tilde R_i$ as an $R_i$-module. Explicitly, this map is just $F \mapsto F \mod (x_1, \ldots, \hat{x}_i, \ldots x_s) \in k\llbracket x_i \rrbracket$.  

Observe that the quotient $\tilde R_i/\mathcal{I}\tilde R_i$ has dimension $a_i$ over $k$. Since the ideals of a power series ring are linearly ordered, it must be isomorphic to $k\llbracket x_i\rrbracket/(x_i^{a_i})$. We conclude that the monomial $x_i^{a_i}$ generates $\mathcal{I}\tilde R_i$ as an $\tilde R_i$-module. In particular, $x_i^{a_i} \in \mathrm{im}(\mathcal I \to \mathcal I \tilde R_i)$ so there exists an $F \in \mathcal I$ with 
$$
F \equiv x_i^{a_i}u(x_i) \mod (x_1, \ldots, \hat{x}_i, \ldots x_s). 
$$
where $u(x_i)$ is a unit in $k\llbracket x_i \rrbracket$. It follows that $F$ can be written $F = x_i^{a_i}u(x_i)+ G$ where $G \in \mathrm{ker}(\mathcal I\to \mathcal I\tilde R_i)$ and $x_i G = 0$. In particular, $v_i(F) = a_i$. As this holds for each $i$, we may apply Lemma \ref{lem:delta} to obtain an inclusion
\[
F^{\underline{\bm{a}} + \underline{\bm{c}}} = \tilde{F}^{\underline{\bm{a}} + \underline{\bm{c}}} \subset \mathcal I
\]
as required. 

On the other hand, since $\mathcal{I}$ has branch-length vector $(a_1, \ldots, a_s)$, then the order of vanishing of $f \in \mathcal{I}$ along the branch $B_i$ cannot have valuation smaller than $a_i$. Applying this to each branch, we see that $\mathcal I \subset \F{\bm{a}}$. 

\end{proof}


\begin{proposition}\label{prop:bounds}
Let $\mathcal I$ be the ideal of a closed subscheme $Z \subset X$ having length $d$ and branch-length vector $\underline{\bm{a}} = (a_1,\ldots, a_s)$. Then we have

\[
- \delta \leq d-\sum_{i=1}^s a_i \leq C - \delta.
\]

\noindent where the second inequality is strict if $(X,0)$ is not smooth. 

\end{proposition}

\begin{proof} By Proposition \ref{prop:inclusions}, there are surjections 
$$
R/F^{\underline{\bm{a}} + \underline{\bm{c}}} \to R/\mathcal I \to R/\F{\bm{a}}
$$
which give us bounds
$$
\dim_k R/\F{\bm{a}} \le d \le \dim_k R/F^{\underline{\bm{a}} + \underline{\bm{c}}}. 
$$

For the upper bound, note that
\begin{align*}
d  &\le \dim_k R/ F^{\underline{\bm{a}} + \underline{\bm{c}}} \\
&= \dim \tilde{R}/\tilde{F}^{\underline{\bm{a}} + \underline{\bm{c}}} - \dim \tilde{R}/R 
\\ &= \sum_{i = 1}^s a_i + \sum_{i = 1}^s c_i - \delta.
\end{align*}
where we have used that $F^{\underline{\bm{a}} + \underline{\bm{c}}} = \tilde{F}^{\underline{\bm{a}} + \underline{\bm{c}}}$. Note however, that we have equality if and only if $J = F^{\underline{\bm{a}} + \underline{\bm{c}}}$. If $(X,0)$ is not smooth, then this is impossible given that necessarily some $c_i > 0$ but $J$ has length profile $\underline{\bm{a}}$. Thus $d \le \sum a_i + C - \delta - 1$ when $(X,0)$ is not smooth.

The lower bound is more delicate as $\dim_k R/F^{\underline{\bm{a}}}$ depends on how the filtration on the normalization $\tF{\bm{a}}$ meets the singularity $R \subset \tilde{R}$, and thus could have complicated combinatorics. To overcome this difficulty, we use an equinormalizable degeneration of $X$ to a toric singularity and observe that in the toric case, the filtration is controlled by monomials. The lower bound is more apparent in this case.

Choose $Y = \Spec A$ a rational curve as in Proposition~\ref{cor:pushout} with normalization $\tilde{A}$ and let $\mathscr{Y} \to \mb{A}^1$ be an equinormalizable degeneration as in Theorem \ref{thm:monomial} to a monomial curve $\mathscr{Y}_0$. By Corollary \ref{cor:monomial}, applied to the idea $\F{\bm{a}} \subset A$, there is a flat family of subschemes $\mathscr{Z} \subset \mathscr{Y}$ of the total space whose nonzero fibers are each isomorphic to $\Spec A/\F{\bm{a}}$. Furthermore, the special fiber $\mathscr{Z}_0$ is identified with $\Spec A_0/F_0$ where $F_0 = \tF{\bm{a}} \cap A_0$ is a monomial ideal of the monomial subring $A_0 \subset A$. 

The algebra $A_0/F_0$ has a monomial basis, specifically consisting of those monomials $x_i^n \in \tilde{A}$ for $0 \le n \le a_i - 1$ that are contained in the toric singularity defined by $A_0$. The number of branches and $\delta$-invariant of $\mathscr{Y}_0$ are the same as that of $Y$, which are in turn the same as that of $(X,0)$. By flatness of the degeneration (Corollary \ref{cor:monomial}) we conclude that 
\begin{align*}
d &\geq \dim_k A/\F{\bm{a}} \\
&= \dim_k A_0/F_0 \\
&\geq \left(\sum_{i =1}^s a_i\right) - \delta,
\end{align*}
as desired.
\end{proof}

\section{Motivic stabilization for the branch-length strata} In this section we prove the key stabilization result from which we deduce rationality. As a corollary we have that the dimensions of the punctual Hilbert schemes stabilize (Corollary~\ref{cor:dim}). This is a generalization of \cite[Theorem 3]{ps}. Following the ideas of \cite{ps}, we use the uniform bounds on an ideal with fixed branch-length vector proved in the previous section to embed the strata $\Hilb^{d, \underline{\bm{a}}}$ as subvarieties of a fixed Grassmannian of $\tilde{R}/N_0$ for an appropriate $R$-submodule $N_0 \subset \tilde{R}$. The image of this embedding lies inside a generalization of the Pfister--Steenbrink variety $\mathcal{M}$ defined in \cite[Section 2]{ps}. 

We first argue that incrementing one entry in the branch-length vector stabilizes once the length on that branch is larger than the conductor. Note that for fixed $d$, the number of possible branch vectors of length $d$ subschemes on a given singularity is finite. Fix an integer tuple $ \underline{\bm{a}}' = (a_1,\ldots, a_{s-1})$ of length $s-1$. Let $\Hilb^{d,\underline{\bm{a}}',e}(X,0)$ denote the stratum of $\Hilb^{d,a_1,\ldots,a_{s-1},e}$ with branch-length vector $(a_1,\ldots,a_{s-1},e)$. 

\begin{theorem}\label{thm: stab}
Let $(X,0)$ be a reduced curve singularity with $s$ branches. Then for $e \geq c_s$, we have equalities
\[
[\Hilb^{d,\underline{\bm{a}}',e}(X,0)] \cong [\Hilb^{d+1,\underline{\bm{a}}',e+1}(X,0)].
\]
in the Grothendieck ring. 
\end{theorem}

\begin{proof}
We introduce the quantity
$$
\alpha := \sum_{k = 1}^{s-1} a_i.
$$
\noindent There is an inclusion of $R$-modules
$$
F^{a_1 + c_1, \ldots, a_{s-1} + c_{s - 1}, e +  c_s} \subset \mathcal{I} \subset F^{a_1, \ldots, a_{s-1}, e} \subset \tilde{R}
$$
\noindent for any $[\mathcal{I}] \in \Hilb^{d, \underline{\bm{a}}', e}$ by Proposition \ref{prop:inclusions}. Moreover, we have inequalities 
$$
d - \alpha + \delta - C \le e \le d - \alpha + \delta
$$
\noindent by Proposition \ref{prop:bounds}. Together, these produce the inclusions
$$
\mf{n}_1 := F^{a_1 + c_1, \ldots, a_{s-1} +  c_{s-1}, d - \alpha + \delta + c_s} \subset \\
\mathcal{I}  \subset F^{a_1, \ldots, a_{s-1}, d-\alpha + \delta - C}
$$

Define 
$$
\epsilon_d := (x_1^{-a_1}, \ldots, x_{s-1}^{-a_{s-1}}, x_s^{-d + \alpha -\delta + C}) \in \mathrm{Frac}(\tilde{R}) = \prod_{i = 1}^s k(\!(x_i)\!)
$$
\noindent where $\mathrm{Frac}(\tilde{R})$ is the total ring of fractions of $\tilde{R}$. Multiplication by $\epsilon_d$ is $R$-module automorphism of $\mathrm{Frac}(\tilde{R})$ and leads to an inclusion  
$$
\mf{n}_0 := F^{c_1, \ldots, c_{s - 1}, C + c_s } \subset \epsilon_d \mathcal{I} \subset \tilde{R} 
$$
\noindent of $R$-modules. Note that $\mf{n}_0$ depends only on which entry of the branch-length vector is varying and not on the specific values of $d,\underline{\bm{a}}'$, or $e$. 

We now compute the dimension 
$$
\dim_k(\epsilon_d \mathcal{I}/\mf{n}_0) = \dim_k(\mathcal{I}/\mf{n}_1).
$$
\noindent By additivity of dimension, the right hand side is equal to 
$$
\dim_k (\tilde{R}/\mf{n}_1) - \dim_k(\tilde{R}/R) - \dim_k(R/\mathcal I) = C. 
$$
\noindent Note that this dimension is independent of $d, \underline{\bm{a}}'$, and $e$. As a consequence we have a well defined map 
$$
\phi_{d,e} : \Hilb^{d, \underline{\bm{a}}', e}(X,0) \to \mathrm{Gr}(C, \tilde{R}/\mf{n}_0)
$$
\noindent given by 
$$
\mathcal{I} \mapsto \epsilon_d \mathcal{I}/\mf{n}_0. 
$$
\noindent This is an embedding into the closed subvariety $\mc{M} \subset \mathrm{Gr}(C,\tilde{R}/\mf{n}_0)$ consisting of those subspaces of $\tilde{R}/\mf{n}_0$ that are $R$-submodules. To see this is a closed subvariety, apply the following observation to the generators of $R$: if $V$ is a vector space and $f : V \to V$ is a linear map, then the set of $f$-stable subspaces of $V$ is closed in the Grassmannian.

Suppose $e \geq c_s$ and let $[\epsilon_d \mathcal{I}/\mf{n}_0] \in \mathrm{im}(\phi_{d,e})$. Consider the $R$-submodule $\xi_s \mathcal{I} \subset \tilde{R}$ where 
$$
\xi_s = (\underbrace{1, \ldots, 1}_{s-1 \text{ times}}, x_s)
$$
\noindent Since $e \geq c_s$, the set $\xi_s \mathcal{I}$ is contained in $R$ and so defines an ideal. Multiplication by $\xi_s$ doesn't change the branch-length vector $\underline{\bm{a}}' = (a_1, \ldots, a_{s-1})$, but increases the length along the $s^{th}$ branch/ It follows that
\[
[\xi_s \mathcal{I}] \in \Hilb^{d+1,\underline{\bm{a}}', e+1}(X,0).
\] 
We have that $\epsilon_{d+1}\xi_s \mathcal{I}/\mf{n}_0 = \epsilon_d \mathcal{I}/\mf{n}_0$. It follows that $[\epsilon_d \mathcal{I}/\mf{n}_0] \in \mathrm{im}(\phi_{d+1,e+1})$ and $\mathrm{im}(\phi_{d,e})$ is contained in $\mathrm{im}(\phi_{d+1,e+1})$. 

On the other hand, suppose $[\epsilon_{d+1} \mathcal{J}/\mf{n}_0] \in \mathrm{im}(\phi_{d+1,e+1})$. By the same argument as above, we see that $\xi_s^{-1} \mathcal{J} \subset R$ is an ideal of length $d$ with branch-length vector $(a_1,\ldots, a_{s-1},e)$ so that $[\epsilon_{d+1} \mathcal{J}/\mf{n}_0] = [\epsilon_d \xi_s^{-1} \mathcal{J}/\mf{n}_0] \in \mathrm{im}(\phi_{d,e})$ and $\mathrm{im}(\phi_{d+1,e+1}) = \mathrm{im}(\phi_{d,e})$. The result follows since the maps $\phi_{d,e}$ are embeddings.  

\end{proof}

\begin{corollary}\label{cor:dim} Let $(X,0)$ be a reduced curve singularity. Then the dimension of $\Hilb^n(X,0)$ stabilizes. \end{corollary} 

\subsection{Conclusion of the proof of the Main Theorem}\label{conclusion} Now we are equipped to conclude the proof of the Main Theorem. As observed in Section \ref{sec:local}, this reduces to the proving Theorem \ref{thm:1}. 

\begin{proof}[Proof of Theorem \ref{thm:1}]
We compute $Z_{0 \subset X}^{\Hilb}(t)$ by stratifying $\Hilb^d(X,0)$ into branch-length strata $\Hilb^{d, \bf{\underline{a}}}(X,0)$. By Proposition \ref{prop:bounds}, for each $d$ there are only finitely many branch-length vectors $\bf{\underline{a}}$ for which $\Hilb^{d, \bf{\underline{a}}}(X,0)$ is non-empty. Thus we may compute as follows.
\begin{align*}
\sum_{d \geq 0} [\Hilb^d(X,0)]t^d = &\sum _{d \geq 0} \sum_{a_1,\ldots,a_s} [\Hilb^{d,a_1,\ldots, a_s}(X,0)]t^d \\ 
= & \sum_{a_1, \ldots, a_s} \sum_{d \geq 0} [\Hilb^{d,a_1,\ldots, a_s}(X,0)]t^d 
\end{align*}

By Theorem \ref{thm: stab}, the Hilbert schemes stabilize under incrementing entries of the branch-length vector, once the lengths are beyond the conductor. Precisely, we have an equality
\[
[\Hilb^{d + k, a_1, \ldots, c_i + k, \ldots a_s}(X,0)] = [\Hilb^{d, a_1, \ldots, c_i, \ldots a_s}(X,0)]
\] 
for any $i$ and $k \geq 0$. Thus for fixed $a_1, \ldots, \hat{a}_i, \ldots, a_s$ we can sum over $a_i$ to get
\begin{align*}
\sum_{a_i\geq 0} \sum_{d \geq 0} [\Hilb^{d,a_1, \ldots, a_s}(X,0)]t^d =& \sum_{a_i = 0}^{c_i - 1} \sum_{d \geq 0} [\Hilb^{d,a_1, \ldots, a_s}(X,0)]t^d \\ +& \frac{1}{1 - t} \sum_{d \geq 0} [\Hilb^{d,a_1, \ldots, c_i, \ldots, a_s}(X,0)]t^d.
\end{align*}

By applying this to each branch and manipulating the summand, we calculate as follows.
\begin{align*}
\sum_{a_1, \ldots, a_s} \sum_{d \geq 0} [\Hilb^{d,a_1,\ldots, a_s}(X,0)]t^d = & \sum_{a_1 = 0}^{c_1 - 1} \ldots \sum_{a_s = 0}^{c_s - 1} \sum_{d \geq 0} [\Hilb^{d, a_1, \ldots, a_s}(X,0)]t^d \\
+ & \sum_{a_1 = 0}^{c_1 - 1} \ldots \sum_{a_{s-1} = 0}^{c_{s-1} - 1}\frac{1}{1 - t} \sum_{d \geq 0} [\Hilb^{d, a_1, \ldots, c_s}(X,0)]t^d \\
+ & \sum_{a_1 = 0}^{c_1 - 1} \ldots \sum_{a_{s-2} = 0}^{c_{s-2} - 1}\frac{1}{(1 - t)^2} \sum_{d \geq 0} [\Hilb^{d, a_1, \ldots, c_{s-1}, c_s}(X,0)]t^d \\
\vdots & \\
+ & \frac{1}{(1 - t)^s} \sum_{d \geq 0} [\Hilb^{d, c_1, \ldots, c_s}(X,0)]t^d
\end{align*}

Finally, by Proposition \ref{prop:bounds}, for each fixed branch-length vector $(a_1, \ldots, a_s)$, the value of $d$ is bounded above and below. Thus the sums over $d$ on the right hand side are all finite so we conclude that the left hand side $Z_{0 \subset X}^{\Hilb}(t)$ is a rational function with denominator $(1 - t)^s$. 

\end{proof}

\section{Extended example: the coordinate axes} 

The bounds in Proposition \ref{prop:bounds} and the stabilization in Theorem~\ref{thm: stab} yield an effective method for computing the Hilbert schemes of many curve singularities. We illustrate this in the following example. 

Let $(X_N,0)$ be the germ at the origin of the coordinate axes $V(\{x_ix_j = 0\}_{i \neq j}) \subset \mb{A}^N$. The normalization $\tilde{X}_N \to X_N$ consists of $N$ branches mapping isomorphically to the branches of $X_N$. On coordinate rings, there is an inclusion
$$
R = k\llbracket x_1, \ldots, x_N \rrbracket/(\{x_ix_j\}_{i \neq j}) \subset \prod_{i = 1}^N k\llbracket x_i \rrbracket = \tilde{R}.
$$
We have $s = N$, $\delta = N - 1$, $\underline{c} = (1,\ldots, 1)$ and $C = N$.  

Let $\underline{a} = (a_1, \ldots, a_s)$ be a branch-length vector and $[\mathcal I] \in \Hilb^{d,\underline{a}}(X_N,0)$. By Proposition~\ref{prop:bounds}, 
$$
0 \le \sum_{i = 1}^s a_i - d \le N - 1. 
$$
Since $c_i = 1$ for all $i$, it follows by Theorem \ref{thm: stab} that the branch-length strata $\Hilb^{d,\underline{a}}(X_N,0)$ stabilize at $a_i = 1$. In this case the above bounds become
$$
0 \le N - d \le N - 1
$$
or $1 \le d \le N$. 

For each $d$, $\Hilb^{d,1, \ldots, 1}(X_N,0)$ embeds into $Gr(N - d + 1, V)$ where 
\[
V = \langle x_1, \ldots, x_N \rangle = F^{1,\ldots, 1}/F^{2, \ldots, 2}.
\] 
Explicitly, the embedding $\phi: \Hilb^{d, 1, \ldots, 1}(X_N,0) \subset Gr(N - d + 1,V)$ is given by 
$$
[\mathcal I] \mapsto [\mathcal I/F^{2, \ldots, 2}] \in Gr(N - d + 1, V)
$$
where we have the containments $F^{2, \ldots, 2} \subset \mathcal I \subset F^{1,\ldots, 1}$ by Proposition \ref{prop:inclusions}. 

As $\underline{c} = (1,\ldots, 1)$, multiplication by $x_i$ acts by $0$ on $V$ so every subspace of $V$ is an $R$-module. In particular, $\mathrm{im}(\phi)$ consists precisely of the locus of subspaces $W \subset V$ which have branch-length vector $(1,\ldots, 1)$. Equivalently, $W \subset V$ must not lie inside any coordinate hyperplane of $V$ under the given coordinates. Denoting the open subset of the Grassmannian parametrizing such $W$ by $Gr(N - d + 1,V)^0$, we conclude that
$$
[\Hilb^{d,1, \ldots, 1}(X_N,0)] = [Gr(N - d + 1,V)^0].
$$

Putting this all together, we obtain
\begin{proposition}\label{prop:axes} Let $(X_N,0)$ be the germ at the origin of the coordinate axes in $\mathbb{A}^N$. Then the Hilbert zeta function is the rational function
$$
Z_{0 \subset X_N}^{\Hilb}(t) =  1 + \frac{1}{(1 - t)^N}\sum_{d = 1}^N \left[Gr(N - d + 1,V)^0\right] t^d.
$$
In particular, $[\Hilb^d(X_N,0)]$ is a polynomial in $\mb{L}$ for all $d$ and $N$. 
\end{proposition}
\begin{proof} This follows from the description of $[\Hilb^{d,1, \ldots, 1}(X_N,0)]$, Theorem \ref{thm: stab} and the computation in Section \ref{conclusion}. Finally, note that $Gr(k,V)^0$ is the complement of the union of $Gr(k,V_i) \subset Gr(k,V)$ where $V_i \subset V$ are the coordinate hyperplanes. By inclusion-exclusion it follows that $[Gr(k,V)^0]$ is a polynomial in $\mb{L}$. 
\end{proof} 

\begin{remark} Zheng \cite[Section 2.3]{xudong} has also performed the same computation for the coordinate axes using different methods. \end{remark} 

\subsection{The axes in three space}

When $N = 3$ we get a particularly pleasant picture. In this case, $V$ is $3$-dimensional and we may compute
\begin{align*}
&\Hilb^{1,1,1,1}(X_3,0) = Gr(3,V)^0 = pt \\
&\Hilb^{2,1,1,1}(X_3,0) = Gr(2,V)^0 = \mb{P}^2 \setminus \{P_1,P_2,P_3\} \\
&\Hilb^{3,1,1,1}(X_3,0) = Gr(1,V)^0 = \mb{P}^2 \setminus (L_1 \cup L_2 \cup L_3)
\end{align*}
where $P_i$ are the distinguished points corresponding to the coordinate hyperplanes in $V$ and $L_i$ are the distinguished lines corresponding to the space of lines in the coordinate hyperplanes of $V$. 

\begin{figure}[h]
\begin{tabular}{lcr}
\begin{tikzpicture}[scale=2]
\coordinate (A) at (.85,-.5);
\coordinate (B) at (-.85,-.5);
\coordinate (C) at (0,1);

\draw[color=black] (A) -- (B) -- (C) -- (A);
\fill[white!70!blue, path fading=east] (A) -- (B) -- (C) -- (A);

\draw[color=black,fill=black] (A) circle (.3mm); 
\draw[color=black,fill=black] (B) circle (.3mm);
\draw[color=black,fill=black] (C) circle (.3mm); 

\end{tikzpicture}& &

\begin{tikzpicture}[scale=2]
\coordinate (A) at (.85,-.5);
\coordinate (B) at (-.85,-.5);
\coordinate (C) at (0,1);

\draw[color=red,dashed] (A) -- (B) -- (C) -- (A);
\fill[white!70!blue, path fading=east] (A) -- (B) -- (C) -- (A);

\end{tikzpicture}
\end{tabular}
\caption{On the left, the stratum $\Hilb^{2,1,1,1}(X_3,0)$ whose closure contains three zero dimensional strata corresponding to twisting $\Hilb^{1,1,1,1}(X_3,0)$ along each of the three branches. On the right, the stratum $\Hilb^{3,1,1,1}(X_3,0)$ whose closure contains coordinate lines inside of $\Hilb^{3,2,1,1}(X_3,0)$ and its permutations.}
\end{figure}
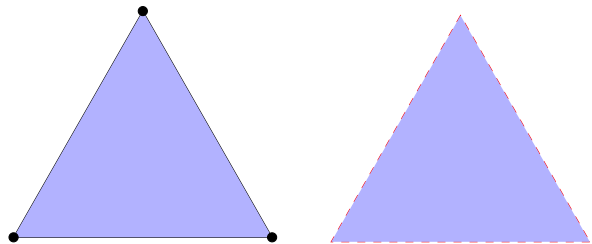

The closure of $\Hilb^{2,1,1,1}(X_3,0)$ contains three strata for branch-length $(2,1,1)$ and its permutations. These are simply the points $P_i$ with the image of $\Hilb^{1,1,1,1}(X_3,0)$ under the identification from Theorem \ref{thm: stab}. Concretely, $P_1$ corresponds to the ideal $(x_1^2, x_2,x_3)$ and similarly for $P_2$ and $P_3$. These are all the possible strata for $d = 2$. 

For $d = 3$ we have the new stratum $\Hilb^{3,1,1,1}(X_3,0)$ as well as the strata coming from $\Hilb^2(X_3,0)$ by twisting along the various branches as in Theorem \ref{thm: stab}. This corresponds to gluing in copies of $\mb{P}^2$ along each of the lines $L_i$ which gives $\Hilb^3(X_3,0)$ as a union of $4$ copies of $\mb{P}^2$ glued along coordinate lines. 

\begin{figure}[h]
\begin{tikzpicture}[scale=2]
\coordinate (A) at (.85,-.5);
\coordinate (B) at (-.85,-.5);
\coordinate (C) at (0,1);
\coordinate (D) at (0,-0.5);
\coordinate (E) at (0.425,0.25);
\coordinate (F) at (-0.425,0.25);

\draw[color=black] (A) -- (B) -- (C) -- (A);
\draw[color=red] (D)--(E)--(F)--(D);
\fill[white!70!blue, path fading=east] (A) -- (D) -- (E) -- (A);
\fill[white!70!blue, path fading=east] (D) -- (E) -- (F) -- (D);
\fill[white!70!blue, path fading=east] (B) -- (D) -- (F) -- (B);
\fill[white!70!blue, path fading=east] (F) -- (E) -- (C) -- (F);

\draw[color=black,fill=black] (A) circle (.3mm); 
\draw[color=black,fill=black] (B) circle (.3mm);
\draw[color=black,fill=black] (C) circle (.3mm); 
\draw[color=black,fill=red] (D) circle (.3mm); 
\draw[color=black,fill=red] (E) circle (.3mm);
\draw[color=black,fill=red] (F) circle (.3mm); 
\end{tikzpicture}
\caption{The reduced Hilbert scheme $\Hilb^3(X_3,0)$ consists of $4$ copies of $\mb{P}^2$ glued along coordinate lines as depicted. The solid shaded strata are $\Hilb^{3,2,1,1}(X_3,0)$ and its permutations. The center stratum is $\Hilb^{3,1,1,1}(X_3,0)$ and the vertices correspond to strata obtained by twisting $\Hilb^{1,1,1,1}(X_3,0)$.}
\end{figure}
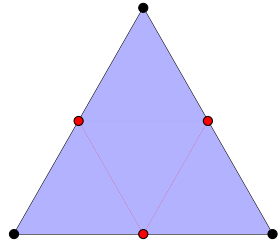

For larger $d$, the strata stabilize and are all obtained from twisting the strata for $d - 1$ along each branch resulting in an arrangement of $\mb{P}^2$'s glued along coordinate lines with dual complex a regular subdivision of the triangle.

\bibliographystyle{siam} 
\bibliography{zeta} 

\begin{thebibliography}{10}

\bibitem{AS11}
{\sc M.~Aganagic and S.~Shakirov}, {\em {Knot homology from refined
  Chern-Simons theory}}, arXiv:1105.5117,  (2011).

\bibitem{Bej18}
{\sc D.~Bejleri}, {\em The {H}ilbert zeta function is constructible in families
  of curves}, http://www.math.brown.edu/~dbejleri/constructible.pdf,  (In
  preparation).

\bibitem{bug}
{\sc R.-O. Buchweitz and G.-M. Greuel}, {\em The {M}ilnor number and
  deformations of complex curve singularities}, Invent. Math., 58 (1980),
  pp.~241--281.

\bibitem{dCL02}
{\sc M.~A.~A. de~Cataldo and L.~Migliorini}, {\em The {C}how groups and the
  motive of the {H}ilbert scheme of points on a surface}, J. Algebra, 251
  (2002), pp.~824--848.

\bibitem{Eis}
{\sc D.~Eisenbud}, {\em Commutative {A}lgebra: with a view toward algebraic
  geometry}, vol.~150, Springer Science \& Business Media, 2013.

\bibitem{FGAx}
{\sc B.~Fantechi, L.~G{\"o}ttsche, L.~Illusie, S.~L. Kleiman, N.~Nitsure, and
  A.~Vistoli}, {\em Fundamental algebraic geometry}, vol.~123 of Mathematical
  Surveys and Monographs, American Mathematical Society, Providence, RI, 2005.
\newblock Grothendieck's FGA explained.

\bibitem{gt}
{\sc R.~Goldin and B.~Teissier}, {\em Resolving singularities of plane analytic
  branches with one toric morphism}, in Resolution of singularities
  ({O}bergurgl, 1997), vol.~181 of Progr. Math., Birkh\"auser, Basel, 2000,
  pp.~315--340.

\bibitem{G90}
{\sc L.~G{\"o}ttsche}, {\em {The Betti numbers of the Hilbert scheme of points
  on a smooth projective surface}}, Mathematische Annalen, 286 (1990),
  pp.~193--207.

\bibitem{G01}
\leavevmode\vrule height 2pt depth -1.6pt width 23pt, {\em On the motive of the
  {H}ilbert scheme of points on a surface}, Math. Res. Lett., 8 (2001),
  pp.~613--627.

\bibitem{GNS15}
{\sc {\'A}.~Gyenge, A.~N{\'e}methi, and B.~Szendr{\H{o}}i}, {\em {Euler
  characteristics of Hilbert schemes of points on simple surface
  singularities}}, arXiv:1512.06848,  (2015).

\bibitem{hartshorne}
{\sc R.~Hartshorne}, {\em Algebraic geometry}, Springer-Verlag, New
  York-Heidelberg, 1977.
\newblock Graduate Texts in Mathematics, No. 52.

\bibitem{kapranov}
{\sc M.~{Kapranov}}, {\em {The elliptic curve in the S-duality theory and
  Eisenstein series for Kac-Moody groups}}, ArXiv Mathematics e-prints,
  (2000).

\bibitem{km}
{\sc K.~{Kaveh} and T.~{Murata}}, {\em {Toric degenerations of projective
  varieties}}, ArXiv e-prints,  (2017).

\bibitem{litt}
{\sc D.~Litt}, {\em Zeta functions of curves with no rational points}, Michigan
  Math. J., 64 (2015), pp.~383--395.

\bibitem{MY14}
{\sc D.~Maulik and Z.~Yun}, {\em Macdonald formula for curves with planar
  singularities}, Journal f{\"u}r die reine und angewandte Mathematik (Crelle's
  Journal), 2014 (2014), pp.~27--48.

\bibitem{MS11}
{\sc L.~Migliorini and V.~Shende}, {\em A support theorem for {H}ilbert schemes
  of planar curves}, J. Eur. Math. Soc. (JEMS), 15 (2013), pp.~2353--2367.

\bibitem{OS12}
{\sc A.~Oblomkov, V.~Shende, et~al.}, {\em {The Hilbert scheme of a plane curve
  singularity and the HOMFLY polynomial of its link}}, Duke Math. J., 161
  (2012), pp.~1277--1303.

\bibitem{ps}
{\sc G.~Pfister and J.~H.~M. Steenbrink}, {\em Reduced {H}ilbert schemes for
  irreducible curve singularities}, J. Pure Appl. Algebra, 77 (1992),
  pp.~103--116.

\bibitem{schwede}
{\sc K.~Schwede}, {\em Obtaining non-normal varieties by pushout}.
\newblock MathOverflow.
\newblock \url{https://mathoverflow.net/q/186650} (version: 2014-11-09).

\bibitem{teiss}
{\sc B.~Teissier}, {\em R{\'e}solution simultan{\'e}e : I - familles de
  courbes}, S{\'e}minaire sur les singularit{\'e}s des surfaces,  (1976-1977),
  pp.~1--10.

\bibitem{stacks-project}
{\sc {The Stacks Project Authors}}, {\em \textit{Stacks project}}.
\newblock \url{http://stacks.math.columbia.edu}, 2017.

\bibitem{xudong}
{\sc X.~Zheng}, {\em The Hilbert Schemes of Points on Singular Varieties and
  Kodaira Non-Vanishing in Characteristic p}, PhD thesis, 2016.

\end{thebibliography}

\end{document}